\documentclass[11pt,leqno]{amsproc}
\usepackage{amsmath,amsthm,amssymb,amsfonts}
\usepackage{graphicx}
\usepackage{color}
\usepackage[active]{srcltx}

\newtheorem{theorem}{Theorem}[section]
\newtheorem{lemma}[theorem]{Lemma}

\newtheorem{proposition}[theorem]{Proposition}
\newtheorem{corollary}[theorem]{Corollary}
\theoremstyle{definition}

\numberwithin{equation}{section}

\def\N{{\mathbb N}}

\def\R{{\mathbb R}}

\def\C{{\mathbb C}}

\def\re{\hbox{\rm Re~}}

\def\llll{{\longrightarrow}}

\def\sep{{ \ \  }}
\def\sem{{\ \ \ \ \ \  }}
\def\seg{{\ \ \ \  \ \  \ \  }}

\def\n3#1{\left\vert  \! \left\vert \! \left\vert \, #1 \, \right\vert \!
  \right\vert \! \right\vert }

\DeclareMathOperator{\rea}{Re}

\def\supp{\text{\rm supp \ }}
\def\com#1{{``#1''}}

\title[The Bishop-Phelps-Bollob\'{a}s property]
{The Bishop-Phelps-Bollob\'{a}s property for operators on $C(K)$}

\author[M. D. Acosta]{Mar\'{\i}a D. Acosta}
\address{Universidad de Granada, Facultad de Ciencias.
Departamento de An\'{a}lisis Matem\'{a}tico, 18071-Granada
(Spain)} \email{dacosta@ugr.es}

\thanks{
The    author was  supported  MTM2012-31755,  Junta de Andaluc\'{\i}a P09-FQM--4911 and
FQM--185. }

\date{2014-05-14}

\begin{document}

\dedicatory{ Dedicated to Rafael Pay\'{a}  on the occasion of his 60th birthday.}

\begin{abstract}
We provide a version for operators of the Bishop-Phelps-Bollob\'{a}s Theorem when  the domain space is the
complex space $C_0(L)$. In fact we prove that the pair $(C_0(L), Y)$  satisfies the
Bishop-Phelps-Bollob\'{a}s property for operators  for every Hausdorff locally compact  space $L$ and any
$\C$-uniformly convex space. As a consequence, this holds for $Y= L_p (\mu)$ ($1 \le p < \infty $).
\end{abstract}

\maketitle

{\large

\section{Introduction}

Bishop-Phelps Theorem  states the denseness of the subset of norm attaining functionals in the (topological)
dual of a Banach space \cite{BP}. Since the Bishop-Phelps  Theorem was proved in the sixties, some
interesting papers  provided  versions of this result for operators. Related to those results, it is worth to
mention  the pioneering work by  Lindenstrauss \cite{Lin}, the somehow surprising result obtained by Bourgain
\cite{Bou} and also  results for concrete classical Banach spaces (see below). In full generality there is no
parallel version of Bishop-Phelps Theorem for operators even if the domain space is $c_0$ \cite{Lin}.
Lindenstrauss also provided some results of denseness of the subset of  norm attaining operators by assuming
some isometric properties either on the domain or on the range space \cite{Lin}.  We mention here two
concrete consequences of these results. If the domain space is $\ell_1$ or the range space is $c_0$, every
operator can be approximated by norm attaining operators. First Lindenstrauss \cite{Lin} and later Bourgain
\cite{Bou}  proved that certain isomorphic assumptions on the domain space (reflexivity or even
Radon-Nikod\'ym property, respectively) implies the denseness of the subset of norm attaining operators in
the corresponding space of linear (bounded) operators. For classical Banach spaces, we only mention some
papers containing  positive results for specific pairs (see for instance \cite{JoWoS}, \cite{Iw},
\cite{Scha}, \cite{Uh}, \cite{APV}) and a few containing counterexamples (see \cite{Scha}, \cite{JoWoP},
\cite{Gow}, \cite{Ag} and \cite{Ac}). The paper \cite{Mar} also answers an old open problem in the topic.

Recently  the paper  \cite{AAGM} dealt with \com{quantitative} versions  of the Bishop-Phelps Theorem for
operators. The motivating result is known nowadays as {\it Bishop-Phelps-Bollob{\'a}s Theorem}  \cite{Bol,
BoDu} and has been a very useful tool to study numerical ranges of operators (see  for instance \cite{BoDu}).
This result can be stated as follows.

 \vskip2mm

 {\it Let  $X$ be a Banach space and $0 < \varepsilon <1$.
Given $x \in B_X$ and $x^\ast \in S_{X^\ast}$  with $\vert 1 - x^\ast(x ) \vert < \dfrac{\varepsilon^2}{4}$,
there are elements $y\in S_X $ and $ y^\ast \in S_{X^\ast}$ such that $y^\ast (y)=1$,  $ \Vert y-x \Vert <
\varepsilon$  and $ \Vert y^\ast -x^\ast \Vert < \varepsilon $.}
 \vskip2mm

Here $X^*$ denotes the (topological) dual of the Banach space $X$ and $S_X$ its unit sphere. We   write $B_X$
to denote the closed unit ball of $X$.

Throughout this paper, for two Banach spaces $X$ and $Y$, $L(X,Y)$ is  the {\it space of linear  bounded
operators} from $X$ into $Y$.  We  recall that the pair $(X, Y)$ {\it has the Bishop-Phelps-Bollob\'as
property for operators} (\emph{BPBp}), if for any $\varepsilon >0$ there exists $\eta(\varepsilon)>0$ such
that for any $T\in S_{\mathcal{L}(X,Y)}$, if $x_0 \in S_X$ is such that $\|Tx_0\|>1-\eta (\varepsilon)$, then
there exist an element $u_0 \in S_X$ and an operator $S \in S_{\mathcal{L}(X,Y)}$ satisfying  the following
conditions:
$$
\Vert Su_0 \Vert =1, \ \ \Vert u_0 - x_0 \Vert <\varepsilon \ \ ~\mbox{and}~ \Vert S-T \Vert
<\varepsilon.
$$

Acosta et all proved that for any space $Y$ satisfying the property $\beta $ of Lindenstrauss, the pair
$(X,Y)$ has the BPBp for operators for every Banach space $X$ \cite[Theorem 2.2]{AAGM}. For  the domain
space, there is no a reasonably general property  implying a positive result. However there are some positive
results in concrete cases.  For instance, there is a characterization of the spaces $Y$ such that the pair
$(\ell_1, Y)$ satisfies the BPBp \cite{AAGM}.  As a consequence of this result, it is known that this
condition is satisfied by finite-dimensional spaces, uniformly convex spaces, $C(K)$ ($K$ is some compact
topological space)   and  $L_1 (\mu)$ (any measure $\mu$). Aron et all showed that the pair $(L_1 (\mu),
L_\infty  ([0,1])) $  has also the BPBp for every $\sigma$-finite measure $\mu$ \cite{ACGM}. This result has
been extended  recently by   Choi et all (see \cite{ChKiLM}).  Some related results for operators whose
domain is  $L_1 (\mu)$  can be also found in \cite{ChKi}, \cite{ABGKM} and \cite{ChKiLM}.

Now we point out results stating that the pair $(X,Y)$ has the BPBp in case that the domain space is $C_0(L)$
(space of continuous functions on a locally compact Hausdorff space $L$ vanishing at infinity). Kim proved
that  in the real  case the pair $(c_0,Y)$ has the BPBp for operators whenever $Y$ is uniformly convex
\cite{Kim}. The paper \cite{ABCCKLLM} contains also a positive result for the pair $(C(K), C(S))$ in the real
case ($K$ and $S$ are compact Hausdorff  spaces). Let us point out  that in the complex case it is not known
yet if the subset of norm attaining operators from $C(K) $ to $C(S)$ is dense in $L(C(K),C(S))$. Very
recently Kim, Lee and Lin \cite{KLL} proved that the pair $(L_\infty (\mu), Y)$ has the BPBp  whenever $Y$ is
a uniformly convex space and $\mu$ is any positive measure. The authors also state the analogous result in
complex case   for  the pairs $(c_0, Y)$ and $(L_\infty (\mu), Y)$  ($\mu$ is any positive measure) whenever
$Y$ is a  $\C$-uniformly convex space. It also holds that the pair $(C(K), Y)$ has the BPBp in the real case
for any uniformly convex space \cite{KL}.

In this  paper we show that  the subspace  of weakly compact operators from $C_0(L)$ into $Y$ satisfies the
Bishop-Phelps-Bollob\'{a}s property for operators  in the complex case, for every  locally compact Hausdorff
space $L$ and for any  $\C$-uniformly convex (complex) space.  Let us notice  that this  is an extension of
the result in \cite{KLL} for the complex case in two ways. First we consider  any space $C(K) $ instead of
$L_\infty (\mu)$ as the domain space and also we consider a strictly more general property on the range
space, namely $\C$-uniform convexity instead of uniform convexity. Our result  extends \cite[Theorem
5.2]{AAGM} in a satisfactory way and the   recent result  in \cite{KLL}  for the case that the domain  space
is $L_\infty (\mu)$. As a consequence, in the complex case the pair $(C(K), L_1 (\mu))$ has the BPBp for
every compact Hausdorff space $K$ and any measure $\mu$.

Let us  recall again that it is not trivial at all to obtain the result in the complex case from the real
case when the domain space is $C(K)$.  As we already pointed out, it is an open problem whether or not  the
subset of norm attaining operators between complex spaces $C(K)$ and $C(S)$ is dense in $L(C(K), C(S))$.
However a positive result  for real $C(K)$ spaces was proved many years ago \cite{JoWoS}.

Let us notice that in case that the range space is  $C(K)$ or more generally  a uniform algebra, the papers
\cite{ACK} and \cite{CGK} provides positive results for the BPBp for the class of Asplund operators.

\section{The result}

Throughout this section  it worths to consider   only {\bf complex} normed spaces. For a complex Banach space
$Y,$ recall that  the {\it $\C$-modulus of convexity} $\delta $ is defined for every $\varepsilon > 0$ by
$$
\delta (\varepsilon )= \inf \bigl\{ \sup \{ \Vert x + \lambda \varepsilon y\Vert -1 : \lambda \in  \C, \vert
\lambda  \vert =1\bigl\}: x,y \in S_Y \bigr\}.
$$
Recall that the Banach space $Y$ is $\C$-{\it uniformly convex } if $\delta (\varepsilon )
> 0$ for every $\varepsilon > 0$ \cite{Glo}.   Every   uniformly convex complex space is $\C$-uniformly
convex and the converse is not true.   Globevnik proved that the complex space $L_1 (\mu)$ is $\C$-uniformly
convex \cite [Theorem 1]{Glo}.

We will denote by $\overline{D}(0,1)$ the closed unit disc in $\C$. Let us notice that for $0 < s < t$ it is
satisfied that $\sup \bigl\{ \Vert x + \lambda sy\Vert : \lambda \in \overline{D}(0,1) \bigr\} \le \sup
\bigl\{ \Vert x + \lambda ty\Vert : \lambda \in \overline{D}(0,1) \bigr\}$. Hence $\delta $ is an increasing
function and $ \delta (t) \le t$ for every $t > 0$.



In what follows $L$ will be  a locally compact Hausdorff topological space and $C_0(L) $ will be the space of
continuous complex valued functions on $L$ vanishing at infinity.

\bigskip

It is convenient to  state  the next trivial result

\begin{lemma}
\label{basic-l}
Assume that $\lambda, w \in \overline{D}(0,1), t \in ]0,1[$  and $\rea w\lambda
> 1-t $. Then $\vert w- \overline{ \lambda} \vert < \sqrt{2t}$.
\end{lemma}


\bigskip


\bigskip




As we already mentioned, the subset of norm attaining operators  between two Banach spaces is not always
dense in the corresponding space of operators in case that the domain space is $C_0(L)$.  Let us notice that
there are examples of spaces $Y$ for which the   subspace of finite-dimensional operators from the space
$\ell_{\infty}^2 $ to $Y$ does not have the Bishop-Phelps-Bollob\'{a}s property (see \cite[Theorem 4.1 and
Proposition 3.9]{AAGM} or \cite[Corollary 3.3]{ACKLM}). For those  reasons some restriction  is needed on the
range space in order to obtain a BPBp result in case that the domain space is $C_0(L)$.

 Schachermayer proved  a Bishop-Phelps result in the real case for the subspace of weakly compact operators
from any space $C_0(L) $ into any Banach space \cite{Scha}. Alaminos et all extended this result to the
complex case \cite{ACKP}.   The last result is one of the tools essentially used in the proof of the main
result. This is our motivation for the next assertion, that might be known, and has interest in itself.

\bigskip


\begin{proposition}
\label{Ja}
Let $Y$ be a $\C$-uniformly convex Banach space and $L$ any locally compact  Hausdorff space. Then every
operator from $C_0(L)$ into $Y$ is weakly compact.
\end{proposition}
\begin{proof}
By the proof of  James  distortion Theorem (see for instance \cite[Proposition 2.e.3]{LinTza})  the space $Y$
cannot contain a copy of $c_0$ (the space of complex sequences converging to zero, endowed with the usual
norm). Otherwise, by considering a convenient multiple of the norm in $Y$, $\n3{}$,  that it is still
$\C$-uniformly convex, one can assume that the usual norm of the copy of $c_0$ ($\Vert \ \Vert$) satisfies
\begin{equation}
\label{norm-equiv}
\alpha \n3{x} \le \Vert x \Vert \le \n3{x}, \ \ \ \forall x \in X,
\end{equation}
for some $\alpha > 0$.  By the proof of \cite[Proposition 2.e.3]{LinTza},  for any $ \varepsilon> 0$  there
is a sequence $(y_n)$ in $Y$  of block basis of the usual basis of $c_0 $ satisfying
$$
\n3{y_n} =1 \ \ \ \ \forall n \in  \N, \ \ \ \ \n3{ \sum _{ k=1}^{\infty } a_n y_n } \le  (1+\varepsilon )^2
\Vert (a_n) \Vert _\infty , \ \ \ \forall (a_n) \in c_0
$$
and $\Vert \sum _{ k=1}^{\infty } a_n y_n \Vert = \Vert (a_n) \Vert _\infty $ for every $(a_n) \in c_0$.
Clearly the above condition contradicts the $\C$-uniform convexity of $Y$.

Now, in view of  Bessaga-Pelczynski selection principle, if the underlying real space of  a complex  space
contains a real space isomorphic to $c_0$, then it contains the complex space $c_0$. So $Y_\R$ does not
contain the real space $c_0$, hence for any compact space $K$,  every (real) operator from the space $C(K)$
(real valued functions) into $Y$ is weakly compact. As a consequence, every operator from the complex space
$C(K)$ into $Y$ is also weakly compact. From here it can be  easily deduced that every operator from $C_0(L)$
into $Y$ is weakly compact,  since  $C_0(L)$ is complemented in the space $C(K)$, being $K$ the Alexandrov
compactification of $L$. Hence every operator from $C_0(L)$ into $Y$ can be extended to an operator from
$C_0(L)$ into $Y$.
\end{proof}

For a locally compact Hausdorff topological  space $L$, we  denote by $\mathcal{B}(L)$ the space of
Borel measurable and bounded  complex valued functions defined on  $L$, endowed with the sup norm. If  $ B
\subset L$ is a Borel measurable set, denote by $P_B$ the projection $P_B :  \mathcal{B}(L) \llll
\mathcal{B}(L) $ given by $ P_B (f) = f \chi _B $ for any $f \in \mathcal{B}(L)$. Of course, in view of Riesz
Theorem, the space $\mathcal{B}(L)$ can be identified  in a natural way  as a subspace of $C_0(L) ^{**}$. As
a consequence, for an operator $T\in L( C_0(L),Y)$ and a Borel set $B \subset L$,  the composition $T^{**}
P_B$ makes sense.

\begin{lemma}
\label{l-UC-P}
Let $Y$ be a $\C$-uniformly convex space with modulus of $\C$-convexity $\delta $.  Let $L$ be a locally
compact Haussdorf  topological space and $A$ a Borel set of $L$. Assume that for some $ 0 <  \varepsilon < 1$
and $T \in S_{L (C_0(L),Y)}$ it is satisfied  $\Vert T^{**} P_A \Vert > 1-   \frac{ \delta (\varepsilon )}{ 1
+ \delta (\varepsilon )} $. Then   $\Vert T ^{**} (I-P_A) \Vert \le \varepsilon.$
\end{lemma}
\begin{proof}
Assume that $T$ satisfies the assumptions of the result.  By Proposition \ref{Ja} $T$ is a weakly compact
operator, so $T^{**} (C_0(L))^{**} \subset Y$ and we consider the subspace $ \mathcal{B}(K) \subset C_0(L)
^{**}$. We write  $\eta = \frac{ \delta (\varepsilon )}{1 + \delta (\varepsilon )}$.  By the assumption,
there exists $ f\in S_{ \mathcal{B}(L)}$ such that $f=P_A (f)$ and $\Vert T^{**} ( f ) \Vert
> 1 - \eta  > 0$. For every $ g  \in B_{ \mathcal{B}(L)}$ it is satisfied that $\Vert  f + (I-P_A) (g) \Vert \le 1$ and so $
\Vert T^{**}(f + \lambda(I-P_A)g ) \Vert \le 1 $ for every $\lambda \in \overline{D }(0,1)$.  That is, for
any $ \lambda \in \overline{D}(0,1)$ we have
\begin{eqnarray*}
 \Bigl \Vert  \frac{T^{**}(f)}{\Vert T^{**}(f) \Vert }  + \lambda \frac{T^{**}(I-P_A)(g) }{\Vert T^{**}(f)
\Vert } \Bigr \Vert   & \le&  \frac{1}{ \Vert T^{**}(f) \Vert} \\
& < &   \frac{1}{1 - \eta  } =  1+\delta (\varepsilon).
\end{eqnarray*}
As a consequence $\Vert T^{**}(I-P_A)(g)  \Vert \le \varepsilon \Vert T^{**}(f) \Vert \le \varepsilon$ and so
$\Vert T^{**} (I-P_A)  \Vert \le \varepsilon  $.
\end{proof}

\bigskip

\begin{theorem}
\label{CK-comlex-UC}
The pair   $(C_0(L),Y)$  satisfies the  Bishop-Phelps-Bollob\'{a}s property for operators for any locally
compact Hausdorff topological space $L$ and any $\C$-uniformly convex space $Y$. Moreover the function $\eta$
appearing in the Definition of BPBp depends only on the modulus of convexity of $Y$.
\end{theorem}
\begin{proof}
Fix $0<\varepsilon<1$ and let $\delta(\varepsilon)$ be the modulus of  $\C$-convexity of $Y$. We denote
$\eta = \dfrac{ \varepsilon^2 \delta \bigl( \frac{\varepsilon}{9}\bigr)^2 } { 10945 \bigl( 1 + \delta \bigl(
\frac{\varepsilon}{9}\bigr) \bigr) ^2 }$ and $ s = \frac{ \eta (2- \varepsilon) \varepsilon^2}{ 2
(\varepsilon^2 + 2 \cdot 12^2 )}$.

Assume that  $T\in S_{L(C_0(L),Y)}$ and $f_0\in S_{C_0(L)}$ satisfy that
$$
\Vert Tf_0\Vert >1-s.
$$
Our goal is to find an operator $S\in S_{L(C_0(L),Y)}$ and $g\in S_{C_0(L)}$ such that
$$
\Vert S(g)\Vert =1, \sem \Vert S-T\Vert <\varepsilon, \sem \text{and} \sem \Vert g- f_0\Vert <\varepsilon.
$$

We can  choose $y_1^*\in S_{Y^*} $ such that
\begin{equation}
\label{y1-T-f0}
\rea y_1^*(Tf_0)=\Vert Tf_0\Vert >1-s .
\end{equation}
 We identify $C_0(L)^*$ with the space $M(L)$
of Borel regular complex measures on $L$ in view of Riesz Theorem. We write $\mu_1= T^* (y_{1}^*)  \in M(L)$.
Since $\mu_1$ is absolutely continuous with respect to its variation $\vert \mu_1 \vert$, by the
Radon-Nikod\'ym Theorem there is a Borel measurable function $g_1 \in \mathcal{B}(L)$ such that $\vert
g_1\vert =1$ and such that
$$
\mu_1 (f)= \int_{L} f g_1 \; d \vert \mu _1 \vert, \seg \forall f \in  C_0(L).
$$

We write $\beta = \frac{\varepsilon^2}{2\cdot 12^2}$ and  denote by $A$ the set given by
$$
A=\bigl\{  t \in L : \re f_0 (t)g_1(t) > 1 - \beta \bigr\}.
$$
By Lemma \ref{basic-l} we have that
\begin{equation}
\label{f0-g1}
\Vert  (f _0- \overline{g_1} ) \chi _A  \Vert_\infty \le    \sqrt{2\beta} =\frac{\varepsilon}{12}.
\end{equation}

Clearly $A$ is also Borel measurable and  we know that
\begin{eqnarray*}
 1-s  &  <  & \rea y_1^*(Tf_0)=\rea  \mu_1 (f_0)  = \re \int_L f_0 g_1 \ d \vert \mu _1 \vert \\
 & = & \re \int_A f_0 g_1 \ d \vert \mu _1 \vert  + \re \int_{L \backslash A}  f_0 g_1 \ d \vert \mu _1 \vert \\
 & \le  & \vert \mu _1 \vert (A) + (1-\beta)   \vert \mu _1 \vert  (L \backslash A)\\
 & =  & \vert \mu _1 \vert (L) - \beta  \vert \mu _1 \vert  (L \backslash A)\\
 & \le  & 1 - \beta   \vert \mu _1 \vert  (L \backslash A).
\end{eqnarray*}
Hence
\begin{equation}
\label{mu1-A}
\vert \mu _1 \vert  (L \backslash A) \le \frac{s}{\beta} = \frac{ \eta (2- \varepsilon)  12^2}{\varepsilon^2
+ 2 \cdot 12^2}.
\end{equation}

By Lusin's Theorem (see for instance \cite[Theorem 2.23]{Ru}) and by the inner regularity of $\mu_1$ there is
a compact set $B \subset A$ such that the restriction of $g_1$ to $B$ is continuous, and $\vert \mu_1 \vert
(A \backslash B) \le  \dfrac{ \varepsilon \eta}{2}$ and so
\begin{equation}
\label{mu1-com-B}
\vert \mu _1 \vert  (L \backslash B) \le \vert \mu _1 \vert  (L \backslash A) + \vert \mu _1 \vert  (A
\backslash B) \le \frac{s}{\beta} + \dfrac{ \varepsilon \eta}{2} .
\end{equation}
From  \eqref{y1-T-f0} and the previous estimate  we obtain
\begin{equation}
\label{mu1-B}
\vert \mu _1 \vert  (B ) = \vert \mu _1 \vert  (L ) - \vert \mu _1 \vert  ( L \backslash B )
> 1 - s -  \frac{s}{\beta} - \dfrac{ \varepsilon \eta}{2} = 1 - \eta .
\end{equation}

Hence
\begin{eqnarray*}
\Vert T^{**} P_B \Vert  & \ge &   \vert \mu_1 \vert (B) \\
& > & 1 - \eta \\
& > & 1 - \frac{ \delta \bigl( \frac{\varepsilon}{9}\bigr) }{ 1+ \delta \bigl( \frac{\varepsilon}{9}\bigr)}
\end{eqnarray*}
By applying Lemma \ref{l-UC-P} we deduce
\begin{equation}
\label{T-TPB}
 \bigl \Vert T^{**}\bigl ( I-P_B \bigr ) \bigr \Vert \le \frac{\varepsilon}{9} .
\end{equation}
By Proposition \ref{Ja} $T$ is a weakly compact operator and so  it is satisfied that $T^{**} (C_0(L) ^{**})
\subset Y$. So we can define the operator $\tilde{S} \in  L ( C_0(L),Y)$  by
$$
\tilde{S}(f)= T^{**}(f \chi_B ) + \varepsilon _1  y_{1}^* \bigl(  T^{**} ( f \chi_B )\bigr) T^{**} (
\overline{g_1}  \chi_B) \seg (f \in C_0(L)),
$$
where $\varepsilon_1=\dfrac{1}{6} \dfrac{\delta \bigl(\dfrac{\varepsilon}{9} \bigr)  }{ 1+ \delta
\bigl(\dfrac{\varepsilon}{9} \bigr)}$.

Let us notice that $\tilde{S}^{**}= \tilde{S}^{**} P_B$ and  we have that
\begin{eqnarray*}
\label{norma-tilde-S}
\Vert \tilde{S} \Vert &  \ge & \vert  y_{1}^* \bigl( \tilde{S}^{**} (\overline{g_1} \chi_B \bigr) \vert    \\
& = &  \bigl \vert  y_{1}^* \bigl( T^{**} (\overline{g_1}  \chi_B)  \bigr) +  \varepsilon_1
\bigl( y_{1}^* \bigl( T ^{**} (\overline{g_1}  \chi_B)  \bigr) \bigl( y_{1}^* \bigl( T ^{**}
(\overline{g_1}
\chi_B)  \bigr)   \bigr \vert      \\
&\ge &   \vert   y_{1}^* \bigl( T ^{**} (\overline{g_1}  \chi_B) \vert \; \vert  1 +
\varepsilon_1 y_{1}^*
\bigl( T ^{**} (\overline{g_1}  \chi_B)  \vert    \\
& \ge & \vert  \mu_1 \vert (B)  \bigl( 1 + \varepsilon_1  \vert  \mu_1 \vert (B)
\bigr)     \\
& > & ( 1-  \eta ) \bigl( 1 + \varepsilon_1 ( 1 - \eta )\bigr) \sem \text{(by \eqref{mu1-B})}
\end{eqnarray*}
As a consequence
\begin{equation}
\label{norma-tilde-S-2}
1 \le 1 - \eta +  \varepsilon_1    \bigl( 1 - \eta \bigr) ^2 \le  \Vert \tilde{S} \Vert \le 1 + \varepsilon
_1,
\end{equation}
and so
\begin{equation}
\label{1-norma-tilde-S}
\bigl \vert  1- \Vert \tilde{S} \Vert \bigr \vert  \le   \varepsilon_1.
\end{equation}

For every $h \in C(B)$, we will denote by $h \chi_B$ the natural  extension of $h$ to $L$, which is a Borel
function on $L$. Let be   $S_1 $ the operator given by
$$
S_1 (h)= \tilde{S} ^{**} (h \chi_B) \seg (h \in C(B)),
$$
which is clearly an operator from $C(B)$ into $Y$.
 Since $ \tilde{ S} ^{**}= \tilde{ S} ^{**}P_B$, it is clear that $\Vert S_1 \Vert = \Vert \tilde{S} \Vert$.
We know that  $B$ is a compact set and $\tilde{S}$ is weakly compact, by \cite[Theorem 2]{ACKP} there is an
operator $S_2 \in L ( C(B), Y)$ and $ h_1 \in S_{ C(B)}$ satisfying that
\begin{equation}
\label{S1-h1}
\Vert  \tilde{S} \Vert= \Vert S_2 \Vert =   \Vert S_2  (h_1) \Vert \sem \text{and} \sem \Vert S_2 - S_1 \Vert
<   \frac{\varepsilon \eta}{2}  .
\end{equation}
We can choose $y_{2}^* \in S_{Y^*} $  such that
\begin{equation}
\label{mu-2-h1}
 y_{2}^* \bigl( S_2 (h_1) \bigr) =  \Vert S_2 \Vert.
\end{equation}
By rotating the elements $h_1$ and $y_{2}^*$ if needed we can also assume that $y_{1} ^*
\bigl( T^{**} \bigl( h_1 \chi_B \bigr) \bigr) \in \R^+_{0}$.
In view of \eqref{S1-h1},  the choice of $y_{2}^*$ and by using that $y_{1} ^* \bigl( T^{**}
\bigl( h_1 \chi_B \bigr) \bigr) \in \R^+_{0}$  we have
\begin{eqnarray*}
\Vert \tilde{S} \Vert  -  \frac{\varepsilon \eta}{2}  & \le &  \re y_{2} ^*  \bigl( S_1 \bigl( h_1 \bigr) \bigr) \\
& = & \re y_{2} ^*  \bigl( \tilde{S} ^{**} \bigl(h_1 \chi_B  \bigr) \bigr) \\
 & = &  \re y_{2} ^*  \bigl( T^{**} (h_1\chi_B ) \bigr) + \varepsilon_1  \re y_{1}^* \bigl(  T ^{**}(h_1 \chi_B
 ) \bigr)
y_{2}^* \bigl( T ^{**}  \bigl( \overline{g_1} \chi _B \bigr) \bigr)  \\
& \le & 1+ \varepsilon_1 \re y_{2}^* \bigl( T ^{**} \bigl( \overline{g_1} \chi _B \bigr)
\bigr).
\end{eqnarray*}
Combining this inequality with  the estimate \eqref{norma-tilde-S} we deduce that
\begin{equation}
\label{y2-T-g1}
\rea \bigl( y_{2} ^* (T^{**}  \bigl(  \overline{g_1 } \chi_B \bigr) \bigr) \ge  \bigl( 1-\eta \bigr) ^2 -
\frac{ \eta  (2 + \varepsilon) }{ 2 \varepsilon_1} .
\end{equation}
As a consequence   we obtain that
\begin{eqnarray*}
\rea y_{2}^* \bigl(  \tilde{S} ^{**}  \bigl( \overline {g_1} \chi_B   \bigr) \bigr) & =& \rea
 y_{2}^* \bigl( T^{**} \bigl( \overline {g_1} \chi_B \bigr) \bigr) +  \varepsilon_1  \re  y_{1}^*(T^{**} \bigl(
\overline{g_1} \chi_B  \bigr)  y_{2}^* \bigl(T^{**} \bigl( \overline{g_1} \chi_B  \bigr) \bigr) \\
 & \ge  &   \bigl( 1-\eta \bigr) ^2 - \frac{ \eta  (2 + \varepsilon) }{ 2 \varepsilon_1}  +\varepsilon_1
 \vert \mu _1 \vert (B) \Bigl( \bigl( 1-\eta \bigr) ^2 - \frac{ \eta  (2 + \varepsilon) }{ 2 \varepsilon_1} \Bigr) \\
 & \ge &   \Bigl( (1-\eta )^2 - \frac{ \eta  (2 + \varepsilon) }{ 2 \varepsilon_1} \Bigr) \Bigl( 1  + \varepsilon_1
 \bigl( 1 - \eta \bigr)  \Bigr) \sem \text{(by \eqref{mu1-B})}.
\end{eqnarray*}
So
\begin{eqnarray*}
\rea y_{2}^* \bigl( S_2  \bigl( \overline {g_1}  _{\vert B} \bigr) \bigr) &  \ge&   \rea y_{2}^* \bigl( S_1
\bigl( \overline {g_1}_ {\vert B} \bigr) \bigr)  - \Vert S_2 -  S_1 \Vert \\
&  \ge&   \rea y_{2}^* \bigl( \tilde{S}^{**}
\bigl( \overline {g_1} \chi_B \bigr) \bigr)  - \Vert S_2 -  S_1  \Vert \\
& \ge &  \Bigl( (1-\eta )^2 - \frac{ \eta  (2 + \varepsilon) }{ 2 \varepsilon_1} \Bigr) \Bigl( 1  +
\varepsilon_1    \bigl( 1 - \eta \bigr)\Bigr)                -  \frac{\eta \varepsilon }{2}  \sem \text{(by \eqref{S1-h1})}\\
& \ge &
 \Bigl( (1-\eta )^2 - \frac{ \eta  (2 + \varepsilon) }{ 2 \varepsilon_1} \Bigr) \Bigl( 1  +
\varepsilon_1    \bigl( 1 - \eta \bigr)\Bigr)                -  \frac{\eta \varepsilon \Vert S_2 \Vert}{2}
 \sem \text{(by \eqref{S1-h1} and \eqref{norma-tilde-S-2})}.
\end{eqnarray*}
 Let us write $R_2 =  \dfrac{S_2}{ \Vert S_2 \Vert}$  and  $\mu_2  = R_2 ^* (y_2^*) \in M(B)$. Let $g_2 =
\dfrac{ d \mu_2} { d \vert  \mu_2 \vert} $  and we can assume that $\vert g_2 \vert =1$. From
the previous inequality, in view of \eqref{S1-h1} and \eqref{norma-tilde-S}  we have that
\begin{eqnarray}
\label{mu2-g1}
\rea y_{2}^* \bigl( R_2  \bigl( \overline {g_1}_{ \vert B} \bigr) \bigr) & \ge & \frac{ \Bigl( (1-\eta )^2 -
\frac{\eta  (2 + \varepsilon )  }{2 \varepsilon_1} \Bigr) \Bigl( 1  + \varepsilon_1
\ \bigl( 1 - \eta \bigr)  \Bigr)}{ \Vert S_2 \Vert}  -\frac{ \eta \varepsilon}{2}  \notag \\
& \ge &  \frac{ \Bigl( (1-\eta )^2 - \frac{\eta  (2 + \varepsilon) }{2\varepsilon_1} \Bigr) \Bigl( 1  +
\varepsilon_1  \bigl( 1 - \eta \bigr)
\Bigr)}{ 1 + \varepsilon_1}    -\frac{ \eta \varepsilon}{2}   \\
& = & 1- \frac{  2 \eta -2 \eta ^2  + \varepsilon_1  \bigl( 1 -  (1-\eta )^3\bigr) + \frac{2\eta + \eta
\varepsilon
}{2\varepsilon_1}  + \frac{\eta \varepsilon}{2}  (2 + \varepsilon_1-\eta) }{  1 + \varepsilon_1} \notag \\
%
%
& > & 1-  6\eta - 2 \frac{\eta }{ \varepsilon_1} - \varepsilon \eta.  \notag
\end{eqnarray}

We consider the measurable set $C$ of $L$ given by
$$
C= \bigl\{ t \in B  : \re \bigl( \overline{g_1} (t)  + h_1 (t) \bigr) g_2 (t) >  2 - \beta \bigr\}.
$$
In view of \eqref{mu-2-h1} and \eqref{mu2-g1} we have that
\begin{eqnarray*}
2  -  6\eta - 2 \frac{\eta }{ \varepsilon_1} - \varepsilon \eta
  &  <  &  \rea  \mu_2 (h_1 + \overline{g_1} _{\vert B} ) \\
& = & \int_{C} \rea (h_1 +\overline{g_1} ) g_2 \; d \vert \mu_2 \vert +   \int_{B\backslash C} \rea (h_1
+\overline{g_1} ) g_2 \; d \vert \mu_2 \vert \\
& \le &    2 \vert \mu_2 \vert (C) +  (2- \beta )  \vert \mu_2 \vert  (B \backslash C)  \\
&  =  &  2 \vert \mu_2 \vert (B)  -  \beta   \vert \mu_2 \vert  (B \backslash C)  \\
&  \le &
  2 -  \beta   \vert \mu_2 \vert  (B \backslash C).
\end{eqnarray*}

Hence
\begin{equation}
\label{mu2-B-C}
\vert \mu_2 \vert  (B \backslash C) \le
\frac{  6 \eta   + 2\frac{\eta}{ \varepsilon_1} + \varepsilon \eta } { \beta }.
\end{equation}

On the other hand, in view of Lemma \ref{basic-l} we have that
\begin{equation}
\label{g1-g2-h1-g2-C}
\Vert \bigl( g_1 - g_2 \bigr) \chi_C \Vert _\infty \le \sqrt{2 \beta}= \frac{\varepsilon}{12}
 \seg \text{and} \seg \Vert \bigl( h_1 - \overline{g_2} \bigr) \chi_C \Vert _\infty \le \sqrt{2 \beta} =
 \frac{\varepsilon}{12}.
\end{equation}
From the previous inequality  and \eqref{f0-g1} it follows that
\begin{equation}
\label{h1-f0-h2-g2-g1-f0}
\Vert \bigl( h_1 - f_0 \bigr) \chi_C \Vert _\infty \le \Vert \bigl( h_1 - \overline{g_2} \bigr) \chi_C \Vert
_\infty  + \Vert \bigl( \overline{g_2}- \overline{g_1} \bigr) \chi_C \Vert _\infty + \Vert \bigl(
\overline{g_1}- f_0 \bigr) \chi_C \Vert _\infty \le \frac{ \varepsilon}{4}.
\end{equation}

By the inner regularity of $\mu_2$ there is a compact set $K_1 \subset C$ such that
\begin{equation}
\label{mu2-C-K1}
\vert \mu_2  \vert (C \backslash K_1 ) < \frac{ \eta \varepsilon}{2}.
\end{equation}

Let us notice that
\begin{eqnarray*}
\Vert R_2^{**} P_ {K_1} \Vert &  \ge &  \Vert y_{2} ^* R_{2}^{**} P_{K_1} \Vert =  \vert \mu_2
\vert (K_1) \\
& = &\vert \mu_2 \vert (B) - \vert \mu_2 \vert (B \backslash C) - \vert \mu_2 \vert(C\backslash K_1)\\
& \ge  &\rea y_{2}^* \bigl( R_{2} \bigl( \overline{g_1} _{\vert B} \bigr) \bigr)  - \vert
\mu_2 \vert (B \backslash C)
 - \vert \mu_2 \vert(C\backslash K_1)\\
& \ge  &\rea y_{2}^*  \bigl( R_{2} \bigl( \overline{g_1} _{\vert B} \bigr)\bigr)   - \vert
\mu_2 \vert (B \backslash C)  - \frac{ \eta \varepsilon}{2} \sem \text{(by \eqref{mu2-C-K1})}\\
& \ge  &  1-  6\eta - 2 \frac{\eta }{ \varepsilon_1} - \varepsilon \eta - \frac{  6 \eta   + 2\frac{\eta}{
\varepsilon_1} + \varepsilon \eta } { \beta }-
 \frac{ \eta \varepsilon}{2}  \sem \text{(by  \eqref{mu2-g1} and \eqref{mu2-B-C})}\\
& >  &  1- 2  \frac{  6 \eta   + 2\frac{\eta}{ \varepsilon_1} + \varepsilon \eta } { \beta }-
 \frac{ \eta \varepsilon}{2}  \\
& > & 1 - \frac{\delta \bigl( \frac{\varepsilon}{9}\bigr)}{1+\delta \bigl( \frac{\varepsilon}{9}\bigr)} > 0.
\end{eqnarray*}
Hence $K_1 \ne \varnothing$.

In view of Lemma \ref{CK-comlex-UC} we obtain
\begin{equation}
\label{R2-K1}
\Vert R_{2} ^{**} (P_B-P_{K_1} ) \Vert \le \frac{\varepsilon}{9}.
\end{equation}

We denote by $T_2$ the element in $L(C_0(L),Y)$ defined  by
$$
T_2 (f) = R_2 (f_{ \vert B}) \seg (f \in C_0(L)).
$$
Clearly it is satisfied that $\Vert T_{2}^{**} (I-P_{K_1} ) \Vert = \Vert R_{2} ^{**}
(P_B-P_{K_1} ) \Vert$ and since $ T_{2} ^{**} (P_B-P_{K_1} ) = T_{2} ^{**} (I-P_{K_1} )P_B$ in
view of \eqref{R2-K1} we obtain
\begin{equation}
\label{T2-B-K1}
\Vert T_{2} ^{**} (P_B-P_{K_1} ) \Vert \le \frac{\varepsilon}{9}.
\end{equation}
We also write $R (f)= T^{**} (f \chi_B) $ for every  $f \in C(B)$ and so we have
\begin{equation}
\label{T2-T-B}
\Vert  \bigl( T_{2} ^{**} - T^{**} \bigr)  P_B  \Vert = \Vert  R_{2} -R \Vert.
\end{equation}
By the definition of $S_1$ we  know that
\begin{equation}
\label{S1-R}
\Vert S_1 -R \Vert \le \varepsilon_1 .
\end{equation}

Since  $K_1 \ne \varnothing$, let us fix $t_0 \in K_1$. Since $K_1 \subset C$,  we have that
$\vert h_1 (t_0) \vert > 1 - \beta > 1 - \frac{\varepsilon}{2}$. So we can choose and open set
$V$ in $B$ such that $t_0 \in V \subset  \bigl\{ t \in B : \vert h_1 (t) \vert > 1 -
\frac{\varepsilon}{2} \bigr\}$ and a function $v  \in C(B)$ satisfying $v(B) \subset [0,1],
v(t_0)=1 $ and $\supp v \subset V$. So there are  functions $h_i \in  C(B)$ ($i=2,3$) such
that
\begin{equation}
\label{h2}
h_2 (t)= h_1 (t) + v(t) \bigl( 1 - \vert h_1 (t) \vert \bigr) \frac{h_1 (t)}{\vert h_1 (t)
\vert} \seg (t \in B).
\end{equation}
and
\begin{equation}
\label{h3}
h_3 (t)= h_1 (t) - v(t) \bigl( 1 - \vert h_1 (t) \vert \bigr)\frac{h_1 (t)}{\vert h_1 (t)
\vert}   \seg (t \in B).
\end{equation}

It is clear that $h_i \in B_{ C(B)}$ for $i=2,3$ and  $h_1 = \dfrac{1}{2} \bigl( h_2 + h_3
\bigr)$.  By using  that the operator $R_2$ attains its norm at $h_1$ we clearly have that
\begin{equation}
\label{R2-h2}
\Vert R_2 (h_2 ) \Vert = 1 \sem \text{and} \sem \vert h_2(t_0 )  \vert =1.
\end{equation}
Since $\supp v \subset  V \subset  \bigl\{ t \in B : \vert h_1 (t) \vert > 1 -
\frac{\varepsilon}{2} \bigr\}$ we obtain for $t \in V$ that
\begin{equation}
\label{h2-h1-V}
\bigl\vert h_2 (t) - h_1 (t) \bigr\vert  \le 1 - \vert h_1 (t) \vert < \frac{\varepsilon}{2}.
\end{equation}
For $t \in B \backslash  V$, $h_2(t) = h_1 (t)$ so $ \Vert h_2 - h_1 \Vert <
\frac{\varepsilon}{2}$. In view of \eqref{h1-f0-h2-g2-g1-f0} we obtain that
\begin{eqnarray}
\label{h2-f0}
\bigl\Vert h_2 - {f_0}_{\vert C} \bigr\Vert  &   \le  & \bigl\Vert h_2  - h_1 \bigr\Vert + \bigl\Vert h_1 -
{f_0}_{\vert C} \bigr\Vert  \notag \\
&  \le &  \frac{\varepsilon}{2}+  \frac{\varepsilon}{4} \\
&  = &   \frac{3\varepsilon}{4}.  \notag
\end{eqnarray}

Since $B \subset L$ is  a compact subset, there is a function $ f_2 \in C_0(L)$ such that it extends the
function $h_2$ to $L$ (see for instance \cite[Corollary 9.15 and Theorem 12.4]{Jam} and \cite[Theorems 17 and
18]{Kel}).
 Since the function $\Phi: \C \llll \C$ given
by $\Phi(z)=z $ if $\vert z \vert \le 1$ and  $\Phi(z)=\frac{z }{\vert z \vert }$ if $\vert z \vert > 1$ is
continuous, by using $\Phi \circ  f_2$ instead of $f_2$ if needed, and the fact that $ h_2 \in S_{ C(B)}$ we
can also assume that  $ f_2 \in S_{ C_0(L)}$. Since $f_2$ is an extension of $h_2$, by using \eqref{h2-f0}
there is an open set $G\subset L$ such that $K_1 \subset G$ and satisfying also that
\begin{equation}
\label{f2-f0}
 \Vert \bigl(  f_2 - f_0 \bigr) \chi _G  \Vert_\infty <   \frac{7\varepsilon}{8}.
\end{equation}

By  Urysohn's Lemma there is a function $u\in C_0(L)$ such that $u(L) \subset [0,1]$, $u_{\vert K_1} =1$ and
$\supp u \subset G$. We define the function $f_3$ by
$$
f_3 = u f_2 + (1-u) f_0,
$$
that clearly belongs to $B_{C_0(L)}$.

Notice also that
\begin{equation}
\label{f3-f2-f0}
f_3(t)= f_2 (t)= h_2(t) \sep \forall t \in K_1, \sem f_3(t)=  f_0(t) \sep \forall t \in L\backslash G
\end{equation}
and
\begin{equation}
\label{f3-G-C}
 \vert f_3 (t) - f_0 (t) \vert = u(t)  \vert f_2 (t) - f_0(t) \vert, \seg \forall t \in G \backslash K_1.
\end{equation}
In view of \eqref{f2-f0} we obtain  that
\begin{equation}
\label{f3-f0}
\Vert f_3 - f_0 \Vert  < \varepsilon.
\end{equation}

We write $\lambda _0 = \overline{h_2 (t_0)}$  and we know that  $\vert \lambda _0 \vert =1$. Define the
operator $S \in L(C_0(L),Y) $ given by
$$
S(f) = R_2 ^{**} \bigl( ( f \chi _{K_1} )_{ \vert B} \bigr)  + \lambda _0  f (t_0)  R_2 ^{**} \bigl( h_2 \chi
_{B \backslash K_1} \bigr)   \seg (f \in C_0(L)).
$$
Since $R_2$ is weakly compact,  $S$  is well-defined.  For every $f \in B_{ C_0(L)}$ we have that $\vert
\lambda _0  f (t_0) \vert  \le  1$ and so
$$
\Vert ( f \chi _{K_1} )_{ \vert B} + \lambda _0  f (t_0)  h_2  \chi _{ B \backslash K_1} \Vert
_\infty  \le 1.
$$
Since $\Vert R_2 \Vert \le 1$, then
$$
\Vert S (f) \Vert =  \Vert  R_2 ^{**} \bigl(  ( f \chi _{K_1} )_{ \vert B}  + \lambda _0 f
(t_0)  h_2 \chi _{ B \backslash K_1} \bigr)  \Vert \le 1.
$$
 It is also satisfied that
\begin{eqnarray*}
 S (f_3)  &= & R_2 ^{**} \bigl(  (f_3 \chi _{K_1})_{\vert B} \bigr)  +
  \lambda _0  f_3 (t_0)  R_2 ^{**} (h_2  \chi _{ B
\backslash K_1})\\
%
%
& = & R_2 ^{**} (h_2) \sem \text{(by \eqref{f3-f2-f0})}\\
& = & R_2 (h_2 )
\end{eqnarray*}
and in view of \eqref{R2-h2}  we obtain  $\Vert  S (f_3)  \Vert = \Vert  R_2 (h_2) \Vert =1$. Hence $S \in
S_{ L (C_0(L),Y)}$ and it attains its norm at $f_3$.  We also know that $\Vert f_3 - f_0 \Vert < \varepsilon
$ by inequality \eqref{f3-f0}. It suffices to check that $S$ is close to $T$.  Indeed we obtain the following
estimate
\begin{eqnarray*}
\Vert S - T \Vert & \le &  \Vert S ^{**} - T^{**} P_B \Vert  + \Vert T^{**} (I- P_B)  \Vert \\
& \le & \Vert  T_2 ^{**} P_{K_1}   - T^{**}  P_B    \Vert +  \Vert R_2^{ **} (P_B-P_{K_1})
\Vert  + \frac{\varepsilon}{9} \sem \text{(by \eqref{T-TPB})}\\
& =& \Vert  \bigl( T_{2}  ^{**} - T^{**} \bigr) P_B    \Vert +   \Vert T_2 ^{ **}
(P_B-P_{K_1}) \Vert  +     \frac{2\varepsilon}{9}    \sem \text{(by \eqref{R2-K1})}    \\
&\le& \Vert  \bigl( T_{2}  ^{**} - T^{**} \bigr) P_B    \Vert +   \frac{\varepsilon}{3} \seg
\text{(by \eqref{T2-B-K1})}\\
&=& \Vert   R_{2}  - R    \Vert +   \frac{\varepsilon}{3} \seg    \text{(by \eqref{T2-T-B})}   \\
&\le& \Vert   R_2 - S_2   \Vert + \Vert   S_2 - S_1   \Vert + \Vert   S_1 - R \Vert
+   \frac{\varepsilon}{3} \\
&\le& \bigl\vert   1- \Vert S_2 \Vert \bigr\vert + \frac{\eta \varepsilon}{2}  + \varepsilon_1+
      \frac{\varepsilon}{3} \seg   \text{(by \eqref{S1-h1} and \eqref{S1-R})}\\
&\le & 2 \varepsilon_1 +  \frac{\eta \varepsilon}{2} +  \frac{ \varepsilon}{3} < \varepsilon \seg \text{(by
\eqref{1-norma-tilde-S} and \eqref{S1-h1})}.
\end{eqnarray*}

\end{proof}

Since  the complex spaces $L_p (\mu)$ ($1 \le p < \infty$) are $\C$-uniformly convex we obtain the following
result:

\begin{corollary}
In the complex case the pair $(C_0(L), L_p (\mu))$  does  have the Bishop-Phelps-Bollob\'{a}s property for
operators for every Hausdorff locally compact space $L$, every positive measure $\mu$ and $1\le p < \infty$.
\end{corollary}

As we already mentioned we extended in a non trivial way a result by S.K. Kim, H.J. Lee and P.K. Lin where
they consider  any (complex) space $L_\infty (\nu)$  as the domain space \cite{KLL}.

\bigskip

{\bf Open problem:} In the real case it is not known  whether or not the pair $(c_0,\ell_1)$ has the BPBp for
operators.

 \vskip20mm

{\bf Acknowledgement:} It is my  pleasure  to thank Miguel Mart\'{\i}n for some comment and to Han Ju Lee for
pointing out some references that simplified the assertion in the main result.

\bigskip

\bibliographystyle{amsalpha}

\end{document}